\documentclass[10pt]{article}
\usepackage{amssymb}
\newtheorem{precor}{{\bf Corollary}}

\newenvironment{cor}{\begin{precor}{\hspace{-0.5
               em}{\bf.\ }}}{\end{precor}}
\newtheorem{precon}{{\bf Conjecture}}

\newtheorem{predefin}{{\bf Definition}}

\newtheorem{preexm}{{\bf Example}}

\newtheorem{preappl}{{\bf Application}}

\newtheorem{prelem}{{\bf Lemma}}

\newenvironment{lem}{\begin{prelem}{\hspace{-0.5
               em}{\bf.\ }}}{\end{prelem}}
\newtheorem{preproof}{{\bf Proof.\ }}

\newenvironment{proof}[1]{\begin{preproof}{\rm
               #1}\hfill{$\blacksquare$}}{\end{preproof}}
\newtheorem{presproof}{{\bf Sketch of Proof.\ }}

\newtheorem{prethm}{{\bf Theorem}}

\newenvironment{thm}{\begin{prethm}{\hspace{-0.5
               em}{\bf.\ }}}{\end{prethm}}
\newtheorem{prealphthm}{{\bf Theorem}}

\newenvironment{alphthm}{\begin{prealphthm}{\hspace{-0.5
               em}{\bf.\ }}}{\end{prealphthm}}
\newtheorem{prealphlem}{{\bf  Lemma}}

\newtheorem{prepro}{{\bf Proposition}}

\newtheorem{preprb}{{\bf Problem}}

\def\conct[#1,#2]{\mbox {${#1} \leftrightarrow {#2}$}}
\def\dconct[#1,#2]{\mbox {${#1} \rightarrow {#2}$}}
\def\deg[#1,#2]{\mbox {$d_{_{#1}}(#2)$}}
\def\mindeg[#1]{\mbox {$\delta_{_{#1}}$}}
\def\maxdeg[#1]{\mbox {$\Delta_{_{#1}}$}}
\def\outdeg[#1,#2]{\mbox {$d_{_{#1}}^{^+}(#2)$}}
\def\minoutdeg[#1]{\mbox {$\delta_{_{#1}}^{^+}$}}
\def\maxoutdeg[#1]{\mbox {$\Delta_{_{#1}}^{^+}$}}
\def\indeg[#1,#2]{\mbox {$d_{_{#1}}^{^-}(#2)$}}
\def\minindeg[#1]{\mbox {$\delta_{_{#1}}^{^-}$}}
\def\maxindeg[#1]{\mbox {$\Delta_{_{#1}}^{^-}$}}

\def\dre[#1,#2,#3]{\mbox {${\cal E}_{_{#3}}(#1,#2)$}}
\def\pdre[#1,#2,#3]{\mbox {${\cal P}_{_{#3}}(#1,#2)$}}
\def\var[#1,#2]{\mbox {${\rm Var}_{_{#1}}(#2)$}}
\def\ls[#1]{\mbox {$\xi^{^{#1}}$}}
\def\hom[#1,#2]{\mbox {${\rm Hom}({#1},{#2})$}}
\def\onvhom[#1,#2]{\mbox {${\rm Hom^{v}}(#1,#2)$}}
\def\onehom[#1,#2]{\mbox {${\rm Hom^{e}}(#1,#2)$}}
\def\core[#1]{\mbox {$#1^{^{\bullet}}$}}
\def\cay[#1,#2]{\mbox {${\rm Cay}({#1},{#2})$}}
\def\cays[#1,#2]{\mbox {${\rm Cay_{s}}({#1},{#2})$}}
\def\dirc[#1]{\mbox {$\stackrel{\rightarrow}{C}_{_{#1}}$}}
\def\cycl[#1]{\mbox {${\bf Z}_{_{#1}}$}}

\begin{document}
\footnotetext[1]{$\ast$This research was in part supported by
a grant from research institute for ICT.}
\begin{center}
{\Large \bf A Note on $d$-Biclique Covers}\\
\vspace*{0.5cm}
{\bf Farokhlagha Moazami$^\ast$, Nasrin Soltankhah$^\ast$ and Shahzad Basiri$^\ddag$}\\
{\it $^\ast$Department of Mathematics} \\
{\it Alzahra University}\\
{\it Vanak Square 19834 Tehran, I.R. Iran}\\
{\tt f.moazami@alzahra.ac.ir}\\
{\tt soltan@alzahra.ac.ir}\\
{\tt $^\ddag$ shahzad\_basiri@yahoo.com}\\
\end{center}
\begin{abstract}
A $d$-{\em biclique  cover} of a graph $G$ is a collection of bicliques of $G$ such that each edge of $G$ is in at least $d$ of the bicliques. The number of bicliques in a minimum $d$-biclique cover of $G$ is called the $d$-{\em
biclique  covering  number} of $G$ and is denoted by $bc_d(G)$. In this paper, we present an upper bound for the $d$- biclique covering number of the lexicographic product of graphs.  Also, we introduce some bounds of this parameter for some graph constructions and obtain the exact value of the $d$-biclique covering number of some graphs.
\begin{itemize}
\item[]{{\footnotesize {\bf Key words:}\ Biclique cover, Fractional biclique cover, Lexicographic product of graphs, Mycielski graph.}}

\item[]{ {\footnotesize {\bf Subject classification:} 05B40.}}
\end{itemize}
\end{abstract}
\section{Introduction and preliminary results}
A biclique cover of a graph $G$ is a collection of
bicliques (complete bipartite subgraphs) of $G$ such that every edge of $G$ is in at least one of these bicliques. The number of bicliques in a minimum biclique cover of $G$ is called the { \em biclique  covering  number} of $G$ and is denoted by $bc(G)$. This measure has been investigated by many researchers and there are several results about this parameter in the literature. We refer the reader to \cite{bc3, bc1, bc2, jukna, bcthes}. Alon, in \cite{bc3}, generalized this definition and defined a { \em bipartite covering of order} $k$ to study $k$-neighborly families of $n$ standard boxes in $R^d$. He showed that the maximum possible cardinality of a $k$-neighborly family of
standard boxes in $R^d$ is precisely the maximum number of vertices of a complete graph that admits a bipartite covering of order $k$ and size $d$. Also, in \cite{haji}, another generalization of this parameter was considered. They defined the $d$-biclique cover of a graph $G$, to investigate the properties of $(r,w;d)$ cover-free families, see~\cite{cff1} for the definition of $(r,w;d)$ cover-free families. A $d$-{\em biclique  cover} of a graph $G$ is a collection of bicliques of $G$ (the same biclique may occur more than once) such that each edge of $G$ is in at least $d$ of the bicliques. The number of bicliques in a minimum $d$-biclique covering of $G$ is called the $d$-{\it
biclique  covering  number} of $G$ and is denoted by $bc_d(G)$. For a graph $G$ a $d$-biclique cover of size $bc_d(G)$ is called an optimal $d$-biclique cover.
In \cite{haji2}, also, the $d$-biclique covering number of the Kneser graph was studied. It was shown that there exists a secure frame proof code if and only if there exists a biclique cover for the Kneser graph.
Secure frameproof codes  can be considered as a tool for digital fingerprinting. For a deeper discussion  for secure frame proof code we refer the reader to \cite{blackburn, frame}. This motivated us to study the properties of this parameter for graphs in general, in this paper. In the rest of this section, we introduce notations and definitions that are used through the paper and give preliminary results. In Section $2$, we obtain an upper bound for lexicographic products of graphs that generalize a bound of \cite{bc1} about biclique covers. In Section $3$, we investigate this parameter for some constructions of graphs such as the join of graphs and Mycielski graph. The result about joins, when $d=1$, improves a result of Watts \cite{bcthes} about biclique covering. In this paper, we only consider finite simple graphs. As usual,  for a graph $G$, let $V(G)$ and $E(G)$ denote its vertex and edge sets, respectively. The complement $\overline{G}$ of $G$ is the simple graph whose vertex set is $V(G)$ and whose edges are the pairs of nonadjacent vertices of $G$.
 $B(G)$ stands for the maximum number of edges among the bicliques of $G$. A biclique $H$ is called maximum whenever $|E(H)|=B(G)$. A simple observation shows that $d|E(G)|\leq bc_d(G)B(G)$. So we have the following lemma.
\begin{lem}\label{frac}For every graph $G$, $d\frac{|E(G)|}{B(G)}\leq bc_d(G)$.
\end{lem}
The $k$-cube, $Q_k$, is a graph whose vertices are the ordered
$k$-tuples of $0^{,}$ s and $1^{,}$ s, two vertices being joined
if and only if they differ in exactly one coordinate. Clearly, the
$k$-cube has $2^k$ vertices and $k2^{k-1}$ edges and is a bipartite graph. Since $Q_k$ is a bipartite graph if we consider the vertices of one part of this graph $d$ times then we have a $d$-biclique cover of size $d 2^{k-1}$. On the other hand, the maximum biclique of $Q_k$, whenever $k\geq 5$, is a star with $k$ edges. So by Lemma \ref{frac} we have $bc_d(Q_k)\geq d 2^{k-1}$. Thus $bc_d(Q_k)= d 2^{k-1}$, for $k\geq 5$. This shows that the lower bound in Lemma \ref{frac} is tight.\\
Another way to approach to the $d$-biclique cover is the fractional version of biclique cover, see \cite{bc1, haji, frac}. A {\it fractional biclique cover} is a function $w$ that assigns to each biclique $G_i$ of a graph $G$ a weight such that $w(G_i)\geq 0$, and for each $e \in E(G)$, $\sum w(G_i)\geq 1$, where this summation is taken over all bicliques  of $G$ that contain $e$. So the {\it fractional biclique cover number} of $G$ is $\min \sum w(G_i)$ where the minimum is taken over all fractional biclique covers, and is denoted by $bc^*(G)$. In the theory of fractional coverings it is a well-known fact that
$$ bc^*(G)=\displaystyle \inf_d \frac{bc_d(G)}{d}=\displaystyle\lim_{d \rightarrow\infty}
 \frac{bc_d(G)}{d}. $$

For every graph $G$, the following inequality was proved in \cite{lovasz, fractional}
 $$bc^*(G)\geq\frac{bc(G)}{1+\ln(B(G))}.$$
 So by the definition of fractional biclique cover and aforementioned inequality
  $$ \frac{d \cdot bc(G)}{1+\ln(B(G))}\leq bc_d(G).$$
Also, the following theorem is well-known.
\begin{alphthm}{\rm\cite{fractional}}
 For every non-empty edge-transitive graph $G$, we have
$$bc^*(G)=\frac{|E(G)|}{B(G)}.$$
Also, there exists a positive integer $d$ such that for every positive integer $t$, $bc^*(G)=\frac{bc_{td}(G)}{td}.$
\end{alphthm}

 So using the fractional biclique cover of a graph $G$, one can see that the lower bound in Lemma \ref{frac} for every edge-transitive graph is tight, for some value of $d$.

Let $K_n$ be the complete graph with $n$ vertices; since $K_n$ is an edge-transitive graph and $K_{\lfloor \frac{n}{2} \rfloor, \lceil \frac{n}{2}\rceil}$ is the maximum biclique of $K_n$, then there exists a positive integer $d$ such that
$$bc_d(K_n)= \left \{\begin{array}{ll}
             \frac{2d(n-1)}{n} &  \,\,\,\ n = 2k \\
             \frac{2dn}{n+1} &  \,\,\,\  n = 2k+1
            \end{array} \right.
$$

Let $C_n$ be a cycle with $n$ vertices. The following lemma easily follows.
\begin{lem} Let $d$ and $n$ be positive integers. Then
$$ bc_d(C_n) = \left \{\begin{array}{ll}
             d &  \,\,\,\,\  n=4, \\
             \frac{n}{2}d &   \,\,\,\ n=2k, \  n \geq 6, \\
             \frac{d}{2}n &  \,\,\,\ n =2k+1, \  d = 2k', \\
              \frac{d-1}{2}n + \lfloor \frac{n}{2}\rfloor +1 & \,\,\  n = 2k+1,\ d = 2k'+1\
            \end{array} \right.$$
\end{lem}

\section{Lexicographic product of graphs}
In this section, we introduce upper bounds of $d$-biclique covering numbers for
lexicographic products of graphs. The lexicographic product  $G[H]$ of graphs $G$ and $H$ has $V(G[H]) =
V (G) \times V (H)$ as its vertex set and  $(x_1, y_1)(x_2, y_2) \in E(G[H])$ if
either $x_1 = x_2$ and $y_1y_2 \in E(H)$, or $x_1x_2 \in E(G)$. In \cite{bc1}, the following two theorems were proved. \begin{alphthm}{\rm \cite{bc1}}\label{a} If there exists a covering of $G$ by $k$ bicliques then there also exists a covering of the lexicographic product $G[\overline{K}_m]$ by $k$ bicliques.
\end{alphthm}
\begin{alphthm}{\rm \cite{bc1}}\label{b} Let there exist a covering of a graph $G$ by $l$ bicliques and a covering of $K_n$ by $k$ bicliques. Then the graph $K_n[G]$ can be covered by $k + l$ bicliques.
\end{alphthm}
A proper vertex coloring is  a function $f: V(G) \rightarrow \{1,2,\ldots, k \}$ which  assign labels or colors to each vertex of a graph such that no edge connects two identically colored vertices. The minimum number of $k$ or colors which the vertices of a graph $G$ may be colored is called the chromatic number and is denoted by $\chi(G)$. An {\it empty graph} on $n$ vertices consists of $n$ isolated vertices without any edges.
The next theorem gives a result that generalizes Theorems~\ref{a} and \ref{b}.

\begin{thm}\label{product} Let $G$ and $H$ be two graphs, then $$bc_{d}(G[H])\leq
bc_{d}(G)+bc_{d}(H)\chi(\overline{G}).$$
\end{thm}
\begin{proof}{ Let  $\{H_1, H_2, \ldots, H_{bc_d(H)} \}$
(resp. $\{G_1, G_2, \ldots, G_{bc_d(G)}\}$ )  be a
$d$-biclique cover of $H$ (resp. $d$-biclique cover of $G$), where each $H_i$ (resp. $G_j$) is a complete bipartite graph. Suppose that  $\{A_i,B_i\}$ (resp. $\{A'_j,B'_j\}$) is a bipartition of the vertex set of $H_i$ (resp. $G_j$). Our objective is to construct a $d$-biclique cover for $G[H]$. Assume that $f: V(G) \rightarrow \{1,2,\ldots, \chi(\overline{G})\}$ is a proper vertex coloring of
$\overline{G}$. Let $f^{-1}(i) = \{i_1,i_2, \ldots, i_{t_i} \}$.  For every $i=1, \ldots, \chi(\overline{G})$ and every $j=1, \ldots, bc_d(H)$, construct the complete bipartite graphs $H_{ij}$ with the vertex set $X_{ij} \cup Y_{ij}$ where $X_{ij}= f^{-1}(i)\times A_j$ and $Y_{ij}= f^{-1}(i)\times B_j$. It is not difficult to see that $H_{ij}$ is a subgraph of $G[H]$. Let $|V(H)|=n$ and $\overline{K_n}$ be the empty graph on the vertices of $H$. In view of the definition of lexicographic product,  one can consider $G_i[\overline{K_n}]$ as a biclique of
$G[H]$, for $i=1, \ldots, bc_d(G)$. We claim that
$$  \{G_1[\overline{K}_n], \ldots, G_{bc_d(G)}[\overline{K}_n] \} \cup  \{ H_{ij}| i= 1, \ldots, \chi(\overline{G})  ,  \ j=1, \ldots, bc_d(H)\} $$
covers every edge of $G[H]$ at least $d$ times. Let $(x_1, y_1)(x_2, y_2)$ be an arbitrary edge of $G[H]$. So either $x_1x_2$ is an edge of $G$ or $x_1=x_2$ and $y_1y_2$ is an edge of $H$. In the first case, there are  $d$ indices  $j_1, \ldots, j_d$ such that $G_{j_1}[\overline{K}_n], \ldots, G_{j_d}[\overline{K}_n]$ contain the edge $(x_1, y_1)(x_2, y_2)$. In the other case $x_1=x_2$, hence,  $f(x_1)=f(x_2)=l$ and $y_1y_2$ is an edge of $H$. So there exist at least $d$ bicliques, say $H_{i_1}, \ldots, H_{i_d}$, which contain $y_1y_2$. In view of the construction of $H_{ij}$, it is not difficult to see that $H_{l,i_1}, \ldots, H_{l,i_d}$ contain $(x_1, y_1)(x_2, y_2)$. Therefore this set is a $d$-biclique cover of size $bc_{d}(G)+bc_{d}(H)\chi(\overline{G})$, as desired.
}
\end{proof}

\section{Join of graphs and Mycielski graph}
The join of two disjoint graphs $G_1$ and $G_2$, which is denoted by $G_1 \vee G_2$, is obtained by taking
a copy of $G_1$, a copy of $G_2$ and joining each vertex in $G_1$ to each vertex in $G_2$.
Watts \cite{bcthes} showed  that $bc(\bigvee_{i=1}^k G_i)\leq \sum_{i=1}^k  bc(G_i)+ k-1$. In the next theorem we present a bound that improves this bound.
\begin{thm}\label{veg}Let $G_1, \ldots, G_k$ be simple graphs. Then $$  \max\{bc_d(G_i)\}_{i=1}^k \leq bc_d(\bigvee_{i=1}^k G_i)\leq \max\{bc_d(G_i)\}_{i=1}^k + bc_d(K_k).$$
\end{thm}
\begin{proof}{  The lower bound is obvious, since every $G_i$ is an induced subgraph of $ \bigvee_{i=1}^k G_i$. Now we want to prove the upper bound. To do this, let $\{ H_{i1},H_{i2}, \ldots, H_{it_i} \}$ be an optimal $d$-biclique cover of $G_i$, i.e., $t_i=bc_d(G_i)$. Without loss of generality assume that $\max\{bc_d(G_i)\}_{i=1}^k=t_1$. For any $1\leq i\leq k$ and $1\leq j\leq t_i$, let $\{X_{ij},Y_{ij}\}$ be a bipartition of the vertex set of $H_{ij}$. Now, we construct a $d$-biclique cover of $\bigvee_{i=1}^k G_i$. For this purpose, for $j=1, \ldots, t_1$, construct the bipartite graph $H_j$ with the vertex set $X_j \cup Y_j$, where $X_j$ is the union of $X_{ij}$'s and $Y_j$ is the union of $Y_{ij}$'s, for all $i$, i.e.,$$ X_j=\bigcup_{i=1}^k X_{ij} \,\,\, {\rm and} \,\,\,\ Y_j=\bigcup_{i=1}^kY_{ij}.$$
Note that some $X_{ij}$'s and $Y_{ij}$'s may be empty. In view of definition of the join operator, $H_j$ is a biclique. The family that contains all $H_j$, covers the edges of $G_i$'s. It remains to cover edges between $G_i$'s. Easily one can cover these edges with a cover of size $bc_d(K_k)$.
}
\end{proof}
It is a well-known fact that the minimum number of bipartite graphs needed
to cover the edges of a graph $G$ is $ \lceil \log \chi(G)\rceil$. In \cite{katona}, Katona and Szemer\'{e}di showed that  $bc(K_n)= \lceil\log(n)\rceil$. One can see that the
inequality  in  Theorem \ref{veg} is the best possible as shown by setting $G_i=K_n$ and $d=1$. This result is also best possible in terms of the parameters $k$ and $d$ as shown in the next corollary.
\begin{cor} Let $H$ be a complete $k$-partite graph. Then $$bc_d(H)=bc_d(K_k).$$
\end{cor}
\begin{proof}{ If we set $G_i=$ the $i^{th}$ partite set of $H$ ($1\leq i \leq k$), then obviously  $H=\bigvee_{i=1}^kG_i$. Since $H$ has a complete graph on $k$ vertices as an induced subgraph, it holds that $bc_d(K_k)\leq bc_d(H)$. On the other hand Theorem~\ref{veg} yields $bc_d(H)\leq bc_d(K_k)$, which proves the corollary.
}
 \end{proof}

\begin{cor}Let $d$ and $n$ be positive integers. Then
$$bc_d(K_n)\leq bc_d(K_{n^2})\leq 2bc_d(K_n).$$
\end{cor}
A star is a tree on $n$ vertices with one vertex having degree $n-1$ and the others having degree $1$. The vertex that has degree $n-1$is called the center of the star. Let $\beta(G)$ denote the minimum number of vertices in a vertex covering. Easily one can see that if $G$ is a $C_4$-free graph then $bc(G)=\beta(G)$. A {\it matching} in a graph is a set of edges without common vertices. The maximum number of edges in a matching of $G$ is called the {\it matching number} and is denoted by $\alpha'(G)$. By  Konig's Theorem it is well-known that for every bipartite graph $ \beta(G)=\alpha'(G)$.
\begin{alphthm}\label{wats}{\rm \cite{bcthes}} Let  $G$  be  a $C_4$-free graph whose $\beta(G)=\alpha'(G)$. Then $$bc^*(G)=bc(G)= \beta(G).$$
\end{alphthm}
Using Theorem~\ref{wats},  we can determine the exact value of the $d$-biclique covering number for a class of graphs containing, among others, all $C_4$-free bipartite graphs.
\begin{thm}\label{equal} Let $G$ be a $C_4$-free graph whose $\beta(G)$ and $\alpha'(G)$ are equal. Then $bc_d(G)= d\beta(G)$ for all positive integers $d$.
\end{thm}
\begin{proof}{ In view of the definition of fractional biclique cover we have  $bc^*(G)=\inf_d \frac{bc_d(G)}{d}$. On the other hand,  $bc_d(G) \leq d \ . \ bc(G)$. So, $$bc^*(G) \leq \frac{bc_d(G)}{d} \leq bc(G).$$ Also, by Theorem~\ref{wats} we have $bc^*(G)=bc(G)= \beta(G)$. Therefore, the desired equality holds.
}
\end{proof}
One can see that the aforementioned theorem does not hold always for any $C_4$-free graph. For instance, consider the graph $C_5$.\\
The {\it Mycielski
graph} of $G$, denoted by $M(G)$, has  $V (G)\cup  [V (G)]'\cup \{u\}$ as the vertex set, and the
edge set $$E(G)\cup \{ xy' \ : \ xy \in E(G) \} \cup \{ y'u \ : y \in V (G) \}.$$
In $M(G)$, the new vertex $u$ is called the {\it root}, and for each $y \in V (G)$, there is a
new vertex, $y'$, called the {\it twin} of $y$. In the next theorem, for a graph $G$ we bound the $d$-biclique covering number of the Mycielski graph and using this theorem we can compute the exact value of the $d$-biclique covering number of some graphs.
Let $\{ G_1, G_2, \ldots, G_t \}$ be a $d$-biclique cover of $G$ where $\{X_i,Y_i\}$ is a bipartition of the vertex set of $G_i$. An optimal $d$-biclique cover of $G$ is called {\it good} whenever there exist $X_1, X_2, \ldots, X_s$ such that every vertex of graph $G$ appears at least $d$ times in the multiset $X_1 \biguplus X_2 \biguplus \cdots \biguplus X_s $. Note that if both $X$ and $Y$ contain $v$, then the multiset $X \biguplus Y$ contains $v$ two times.
\begin{thm}\label{mis}Let $G$ be a graph, then $$bc_{2d}(G)\leq bc_d(M(G)) \leq 2bc_d(G)+d.$$
Also, if $G$ has a good $d$-biclique cover, then
$$bc_{2d}(G)\leq bc_d(M(G)) \leq 2bc_d(G).$$
\end{thm}
\begin{proof}{ In order to prove the lower bound, we present a $2d$-biclique cover of $G$, using an optimal $d$-biclique cover of $M(G)$. Set $\{ G_1, \ldots, G_t\}$  as a $d$-biclique cover of $M(G)$ such that $t=bc_d(M(G))$. Suppose that $G_i$ is a complete bipartite graph where $\{X_i,Y_i\}$ is a bipartition of its vertex set. Assign to the pair $\{X_i,Y_i\}$, the pair $\{V_i, W_i\}$ as follows:
$$\begin{array}{c}
  V_i=\{v_j \ | \ v_j \in X_i \ {\rm or} \ v'_j \in X_i \}\setminus \{u\}, \\
W_i=\{v_j \ | \ v_j \in Y_i \ {\rm or} \ v'_j \in Y_i \}\setminus \{u\}.
\end{array}
$$
   Since $G_i$ is a biclique we have $V_i \cap W_i =\varnothing $. Now for every $1\leq i \leq t$ construct the complete bipartite graph $H_i$ as a subgraph of $G$ where $\{V_i, W_i\}$ is a bipartition of its vertex set. Let ${\cal C}= \{ H_1, \ldots, H_t\}$. If there exists some $i$ such that $X_i= \{u\}$ or $Y_i= \{u \}$ then the corresponding $H_i$ is an empty graph and we can remove it from ${\cal C}$. We claim that ${\cal C}$ is a $2d$-biclique cover of $G$. To see this, assume that $v_kv_l$ is an arbitrary edge of $G$. Consider two edges $v'_kv_l$ and $v_kv'_l$ of $M(G)$. It is clear that $G_i$ has either $v'_kv_l$ or $v_kv'_l$. Since $\{ G_1, \ldots, G_t\}$ is a $d$-biclique cover of $M(G)$, there exist at least $i_1, \ldots, i_d$ and $i'_1, \ldots, i'_d$ such that $G_{i_j}$ $(j=1, \ldots, d)$ has $v'_kv_l$ and  $G_{i'_j}$ $(j=1, \ldots, d)$ has $v_kv'_l$. Easily, one can see that $H_{i_j}$ and $H_{i'_j}$, for $j=1, \ldots, d$, have edge $v_kv_l$. So we have a $2d$-biclique cover of size $ bc_d(M(G))$ for $G$. We proceed to show the upper bound in the first statement. We will give a $d$-biclique cover for $M(G)$, using a $d$-biclique cover of $G$. Assume that $ \{G_1, \ldots, G_t\}$  is an optimal $d$-biclique cover of $G$. Suppose that $G_i$ has $\{X_i,Y_i\}$ as a bipartition of its vertex set. Using the notation of Mycielski graphs, let $G'_i$ (resp. $G''_i$) be a complete bipartite graph where has  $\{X'_i\cup X_i,Y_i\}$ (resp. $\{X_i,Y'_i \cup Y_i\}$) as a bipartition of its vertices, where $X'_i$ is the set of twins of $X_i$. Let $H$ be  the closed neighborhood of $u$ in $M(G)$, which is a star with $u$ as its center. It is not difficult to see that $ \{ G'_1, \ldots, G'_t \} \biguplus \{ G''_1, \ldots, G''_t \}$ with $d$ copies of $H$ is a $d$-biclique cover for $M(G)$. Hence $bc_d(M(G)) \leq 2bc_d(G)+d$. Our next goal is to eliminate stars from the aforementioned cover. So assume that we have a good $d$-biclique cover. Without loss of generality assume that $X_1 \biguplus \ldots \biguplus X_s$ is the set that has every vertex of $G$ at least $d$ times. Let $G'_i$ be a complete bipartite graph which has  $\{X'_i,Y_i \cup \{u \}\}$ as a bipartition of its vertices and let $G''_i$ be a complete bipartite graph where  $\{X_i,Y'_i \cup Y_i\}$ is  a bipartition of its vertices. Easily one can see that $ \{ G'_1, \ldots, G'_t \} \biguplus \{ G''_1, \ldots, G''_t \}$ is a $d$-biclique cover for $M(G)$.

}
\end{proof}
\begin{cor}Let $G$ be a $C_4$-free graph whose $\beta(G)= \alpha'(G)$. If $G$ has a minimal vertex cover such that the induced graph on this set of vertices has no isolated vertex then $$bc_d(M(G)) = 2d\beta(G). $$
\end{cor}
\begin{proof}{ Suppose that  $G$ has a vertex cover with the aforementioned
property. It follows easily that $G$ has a good $1$-biclique cover. Consider this $1$-biclique cover $d$ times, then we have a good $d$-biclique cover. Therefore, the corollary follows directly from Theorem \ref{mis}.
}
\end{proof}
If $d$ is an even number then we can construct a good $d$-biclique cover for $C_n$. So by Theorem \ref{mis} we have the following corollary.

\begin{cor} Let $n>4$.  Then for every positive integer $d$,
$$bc_{2d}(M(C_n))=2dn .$$
\end{cor}

{\bf Acknowledgment}\\ We would like to acknowledge Professor
Hossein Hajiabolhassan  for his helpful comments. Also, the authors wish to thank the anonymous referee who drew their attention to Theorems~\ref{equal} and \ref{wats} and for helpful comments.





\end{document}